\documentclass[12pt]{amsart}
 \usepackage{amscd,amsmath,amsthm,amssymb}
 \usepackage{pstcol,pst-plot,pst-3d}%\usepackage[T1]{fontenc}
 \psset{unit=0.7cm,linewidth=0.8pt,arrowsize=2.5pt 4}
\newpsstyle{fatline}{linewidth=1.5pt}
\newpsstyle{fyp}{fillstyle=solid,fillcolor=verylight}
\definecolor{verylight}{gray}{0.97}
\definecolor{light}{gray}{0.93}
\definecolor{medium}{gray}{0.82}
\definecolor{dark}{gray}{0.72}
 \def\F{{\mathcal F}}
 \def\C{{\mathcal C}}
 \def\D{{\mathcal D}}
 \def\B{{\mathcal B}}
 \def\A{{\mathcal A}}
 \def\T{{\mathcal T}}

 \def\ab{{\bold a}}
 \def\bb{{\bold b}}

 \def\cb{{\bold c}}

 \def\opn#1#2{\def#1{\operatorname{#2}}} % to make operators
 \opn\chara{char} \opn\length{\ell} \opn\pd{pd} \opn\rk{rk}
 \opn\projdim{proj\,dim} \opn\injdim{inj\,dim} \opn\rank{rank}
 \opn\depth{depth} \opn\grade{grade} \opn\height{height}
 \opn\embdim{emb\,dim} \opn\codim{codim}
 
 \opn\Tr{Tr} \opn\bigrank{big\,rank}
 \opn\superheight{superheight}\opn\lcm{lcm}
 \opn\trdeg{tr\,deg}%\emph{
 \opn\reg{reg} \opn\lreg{lreg} \opn\ini{in} \opn\lpd{lpd}
 \opn\size{size} \opn\sdepth{sdepth}
 \opn\link{link}\opn\fdepth{fdepth}\opn\lex{lex}
 \opn\div{div} \opn\Div{Div} \opn\cl{cl} \opn\Cl{Cl}
 \opn\Spec{Spec} \opn\Supp{Supp} \opn\supp{supp} \opn\Sing{Sing}
 \opn\Ass{Ass} \opn\Min{Min}\opn\Mon{Mon}
 \opn\Ann{Ann} \opn\Rad{Rad} \opn\Soc{Soc}
 \opn\Im{Im} \opn\Ker{Ker} \opn\Coker{Coker} \opn\Am{Am}
 \opn\Hom{Hom} \opn\Tor{Tor} \opn\Ext{Ext} \opn\End{End}
 \opn\Aut{Aut} \opn\id{id}
 
 \opn\nat{nat}
 \opn\pff{pf}%   \pf exists already
 \opn\Pf{Pf} \opn\GL{GL} \opn\SL{SL} \opn\mod{mod} \opn\ord{ord}
 \opn\Gin{Gin} \opn\Hilb{Hilb}\opn\sort{sort}
 \opn\aff{aff} \opn\con{conv} \opn\relint{relint} \opn\st{st}
 \opn\lk{lk} \opn\cn{cn} \opn\core{core} \opn\vol{vol}
 \opn\link{link} \opn\star{star}\opn\lex{lex}\opn\set{set}
 \opn\gr{gr}
 
 \def\pot#1#2{#1[\kern-0.28ex[#2]\kern-0.28ex]}

 \opn\dirlim{\underrightarrow{\lim}}
 \opn\inivlim{\underleftarrow{\lim}}
 \let\union=\cup
 \let\sect=\cap

 \let\Union=\bigcup
 
 \let\Dirsum=\bigoplus
 
 \def\Implies{\ifmmode\Longrightarrow \else
         \unskip${}\Longrightarrow{}$\ignorespaces\fi}
 \def\implies{\ifmmode\Rightarrow \else
         \unskip${}\Rightarrow{}$\ignorespaces\fi}
 \def\iff{\ifmmode\Longleftrightarrow \else
         \unskip${}\Longleftrightarrow{}$\ignorespaces\fi}

 \let\:=\colon
 \newtheorem{Theorem}{Theorem}[section]
 \newtheorem{Lemma}[Theorem]{Lemma}
 \newtheorem{Corollary}[Theorem]{Corollary}
 \newtheorem{Proposition}[Theorem]{Proposition}

 \let\epsilon\varepsilon
 \let\kappa=\varkappa
 \def\qed{\ifhmode\textqed\fi
       \ifmmode\ifinner\quad\qedsymbol\else\dispqed\fi\fi}
 \def\textqed{\unskip\nobreak\penalty50
        \hskip2em\hbox{}\nobreak\hfil\qedsymbol
        \parfillskip=0pt \finalhyphendemerits=0}
 \def\dispqed{\rlap{\qquad\qedsymbol}}
 
 \opn\dis{dis}
 \def\pnt{{\raise0.5mm\hbox{\large\bf.}}}
 
 \opn\Lex{Lex}

\begin{document}

 \title {Squarefree vertex cover algebras}

 \author {Shamila Bayati and Farhad Rahmati}

 \address{Shamila Bayati, Faculty of Mathematics and Computer Science,
Amirkabir University of Technology (Tehran Polytechnic), 424 Hafez Ave., Tehran
15914, Iran}\email{shamilabayati@gmail.com}

\address{Farhad Rahmati, Faculty of Mathematics and Computer Science,
Amirkabir University of Technology (Tehran Polytechnic), 424 Hafez Ave., Tehran
15914, Iran} \email{frahmati@cic.aut.ac.ir}

 \begin{abstract}
In this paper we introduce squarefree vertex cover algebras. We study the question when these algebras coincide with the ordinary vertex cover algebras and when these algebras are standard graded. In this context we exhibit a duality theorem for squarefree vertex cover algebras.
 \end{abstract}


 \maketitle

 \section*{Introduction}

 The  study of squarefree vertex cover algebras was originally the  motivation  to better understand the Alexander dual of the facet ideals of the skeletons of a simplicial complex. In the special case of  the simplicial complex $\Delta_r(P)$ whose facets correspond to sequences $p_{i_1}\leq p_{i_2}\leq \ldots \leq p_{i_r}$ in a finite poset $P$, it was shown in \cite{VHF} that $I(\Delta_{r}(P)^{(k)})^\vee=(I(\Delta_{r}(P))^\vee)^{\langle d-k\rangle}$; see Section~1 for a detailed explanation of the formula.

 The question arises whether this kind of duality is  valid for more general simplicial complexes. Considering this problem,  it turned out  that this is indeed the case for a pure  simplicial complex $\Delta$ on the vertex set $[n]$, provided  a certain algebra attached to $\Delta$ is standard graded. The algebra in question, which we now  call the squarefree vertex cover algebras of $\Delta$, denoted by $B(\Delta)$, is defined as follows: let  $K$ be a field and $S=K[x_1,\ldots,x_n]$ be the polynomial ring over $K$ in the variables $x_1,\ldots,x_n$. Then $B(\Delta)$ is the graded $S$-algebra generated by the  monomials  $x^\cb t^k\in S[t]$ where the $(0,1)$-vector $\cb$ is a $k$-cover of $\Delta$ in the sense of \cite{HHT}. Thus in contrast to the vertex cover algebra $A(\Delta)$, introduced in \cite{HHT}, whose generators correspond to all $k$-covers, $B(\Delta)$ is generated only by the  monomials corresponding to squarefree $k$-covers, called binary $k$-covers  in \cite{DV}. In particular, $B(\Delta)$ is a graded $S$-subalgebra of $A(\Delta)$ whose generators in degree $1$ coincide.

The graded components of $B(\Delta)$ are of the form $L_k(\Delta)t^k$, where $L_k(\Delta)$ is a monomial ideal whose squarefree part, denoted by $L_k(\Delta)^{sq}$, corresponds to the squarefree $k$-covers. The above mentioned duality is a consequence of   the following duality
\begin{eqnarray}
\label{volleyball}
 L_j(\Delta^{(d-i)})^{sq}=L_i(\Delta^{(d-j)})^{sq}, \quad 1\leq i,j\leq d,
\end{eqnarray}
inside $B(\Delta)$,  which is valid for any pure simplicial complex of dimension $d-1$, no matter whether $B(\Delta)$ is standard graded or not.

This result is  a simple consequence of Theorem~\ref{dual} where it is shown that the squarefree $k$-covers of $\Delta$  correspond to the vertex covers of the $(d-k)$-skeleton  $\Delta^{(d-k})$ of $\Delta$.

The duality  described in (\ref{volleyball}) yields the  desired generalization of \cite[Theorem 1.1]{VHF}, and we obtain in Corollary~\ref{duality}
\[
 I(\Delta^{(k)})^\vee=(I(\Delta)^\vee)^{\langle d-k\rangle} \quad\text{for all} \quad k,
\]
if and only if $B(\Delta)$ is standard graded. Therefore it is of interest to know when $B(\Delta)$ is standard graded.

The starting point of our investigations has been the formula  $I(\Delta_{r}(P)^{(k)})^\vee=(I(\Delta_{r}(P))^\vee)^{\langle d-k\rangle}$. As explained  before this implies that $B(\Delta_r(P))$ is standard graded. As a last result in Section~1, we show that even  the algebra $A(\Delta_r(P))$ is standard graded.

In Section 2 we compare the algebras $A(\Delta)$ and $B(\Delta)$, and discuss the following questions:
\begin{itemize}
\item[(1)]   When is $B(\Delta)$ standard graded and when does this imply that $A(\Delta)$ is standard graded?
\item[(2)] When do we have that  $B(\Delta)=A(\Delta)$?
\end{itemize}
In general $B(\Delta)$ may be standard graded, while $A(\Delta)$ is not standard graded, as an example shows which was communicated to the authors by Villarreal; see Section~2.

On the other hand,  quite often it happens that  $A(\Delta)$ is  standard graded if $B(\Delta)$ is standard graded.  In Proposition~\ref{standardgraph} we show that for any $1$-dimensional simplicial complex $\Delta$, the algebra $B(\Delta)$ is standard graded if and only if $A(\Delta)$ is standard graded. We also show in Proposition \ref{coveringIdeal} that the same statement holds true  when the facet ideal of $\Delta$ is the covering ideal of a graph.  In Theorem~\ref{no-odd}, it is shown that this is also the case for all  subcomplexes of $\Delta$ if and only if $\Delta$ has no special odd cycles. In the  remaining part of Section~2 we present cases for which we have $B(\Delta)=A(\Delta)$.  The $1$-dimensional simplicial complexes with this property are classified in Proposition~\ref{graphequality}. A classification of simplicial complexes $\Delta$ of higher dimension with  $B(\Delta)=A(\Delta)$ seems to be unaccessible for the moment. Thus we consider simplicial complexes in higher dimensions which generalize the concept of a graph. Roughly speaking, we replace the vertices of a graph by simplices of various dimensions. This graph is called the intersection graph of $\Delta$.   The main result regarding this class of simplicial complexes is formulated in Theorem~\ref{str-intersec-prop}, where the criterion for the equality $A(\Delta)=B(\Delta)$ is given in terms of the intersection graph of $\Delta$.

The last section of this paper is devoted to study the vertex cover algebras of shifted  simplicial complexes. A   simplicial complexes $\Delta$ is shifted if its set of facets  is a Borel sets $\B$. When the set of facets of $\Delta$ is principal Borel we have  $B(\Delta)=A(\Delta)$ as shown in Theorem~\ref{borel-generators}. Since all skeletons of such simplicial complexes also correspond to principal Borel sets, one even has that  $B(\Delta^{(i)})=A(\Delta^{(i)})$ for all~$i$.

The squarefree monomial ideal $(\{x_F\:\; x_Ft\in A_1(\Delta)\})$ generated by the degree~1 elements of $A(\Delta)$ is the Alexander dual $I(\Delta)^\vee$ of $I(\Delta)$. Francisco, Mermin and Schweig showed that this ideal is again a squarefree Borel ideal, and they give precise Borel generators in the case that $\B$ is principal Borel. We generalize this result in Theorem~\ref{B-generators}, by showing that  the squarefree part of the ideal $(\{x_F\:\; x_Ft^k\in A_k(\Delta)\})$ is again a squarefree Borel ideal whose generators can be explicitly  described when $\B$ is principal Borel. It turns out that in this case  the $S$-algebra $A(\Delta)$ may have  minimal generators in the degrees up to $\dim\Delta+1$. In Proposition~\ref{higher-generator}, we present a necessary and sufficient condition such that this maximal degree achieved.
\medskip

We would like to thank Professor Villarreal for several useful comments and for drawing our attention to related work on this subject.

 \section{Duality}

We first fix some notation and recall some basic concepts regarding simplicial complexes.

Let $\Delta$ be a simplicial complex of dimension $d-1$ on the vertex set $[n]$. We denote by $\F(\Delta)$ the set of facets of $\Delta$.

The $i$-skeleton $\Delta^{(i)}$ of $\Delta$ is defined to be the simplicial complex whose faces are those of $\Delta$ with $\dim F\leq i$. Observe that  $\Delta^{(i)}$  is a pure simplicial complex with $\F(\Delta^{(i)})=\{F\in \Delta\:\; \dim F=i\}$ if $\Delta$ is pure.

\medskip
The Alexander dual $\Delta^\vee$ of $\Delta$ is the simplicial complex with
\[
\Delta^\vee=\{[n]\setminus F\:\; F\not\in \Delta\}.
\]
One has $(\Delta^\vee)^\vee =\Delta$.

\medskip
Let $K$ be a field and $S=K[x_1,\ldots,x_n]$ the polynomial ring over $K$ in the variables $x_1,\ldots,x_n$. The Stanley-Reisner ideal of $\Delta$ is defined to be
\[
I_\Delta=(\{x_F\:\; F\subseteq  [n], F\not\in\Delta\}).
\]
Here $x_F=\prod_{i\in F}x_i$ for $F\subseteq [n]$.

For $F\subseteq [n]$, we denote  by $P_F$ the monomial prime ideal  $(\{x_i\:\; i\in F\}$. Let   $$I_\Delta=P_{F_1}\sect \cdots \sect P_{F_m}$$ be  the irredundant  primary decomposition of $I_\Delta$, then
$I_{\Delta^\vee}$ is minimally  generated by $x_{F_1},\ldots,x_{F_m}.$

Now let $I\subseteq S$ be an arbitrary squarefree monomial ideal.  There is a unique simplicial complex $\Delta$ on $[n]$ such that  $I=I_\Delta$. We set $I^\vee =I_{\Delta^\vee}$. It follows that if $I=P_{F_1}\sect \cdots \sect P_{F_m}$, then $I^\vee=(x_{F_1},\ldots,x_{F_m})$, and if $J=(x_{G_1},\ldots,x_{G_r})$, then $J^\vee=P_{G_1}\sect \cdots \sect P_{G_m}$.

\medskip
Let $k$ be a nonnegative  integer. A {\em $k$-cover} or a {\em cover of order $k$} of $\Delta$  is a nonzero vector $\mathbf{c}=(c_1,\ldots,c_n)$ whose entries are nonnegative integers such that $\sum_{i\in F}c_i\geq k$ for all $F\in \F(\Delta)$. We denote the set $\{i\in [n]\:\; c_i\neq 0\}$ by $\supp(\cb)$. The $k$-cover $\mathbf{c}$ is called  {\em squarefree } if $c_i\leq 1$ for all $i$. A $(0,1)$-vector $\mathbf{c}$ is a squarefree cover of $\Delta$ with positive order if and only if $\supp(\cb)$ is a vertex cover of $\Delta$. A $k$-cover $\mathbf{c}$ of $\Delta$ is called {\em decomposable} if there exist an $i$-cover $\mathbf{a}$ and a $j$-cover  $\mathbf{b}$ of $\Delta$ such that $\mathbf{c}=\mathbf{a}+\mathbf{b}$ and $k=i+j$. Then $\cb=\ab+\bb$ is called a {\em decomposition} of $\cb$. A $k$-cover of $\Delta$ is {\em indecomposable} if it is not decomposable.

The $K$-vector space spanned by the monomials $x^{\cb}$ where $\cb$ is a $k$-cover, denoted by $J_k(\Delta)$, is an ideal. Obviously one has $J_k(\Delta)J_\ell(\Delta)\subseteq J_{k+\ell}(\Delta)$ for all $k$ and $\ell$. Therefore
\[
A(\Delta)=\Dirsum_{k\geq 0}J_k(\Delta)t^k\subseteq S[t]
\]
is a graded $S$-subalgebra of the polynomial ring $S[t]$ over $S$ in the variable $t$. The $S$-algebra $A(\Delta)$ is called the vertex cover algebra  of $\Delta$. This algebra is minimally generated by  the monomials $x^{\mathbf{c}}t^k$ where $\mathbf{c}$ is an indecomposable $k$-cover of $\Delta$ with $k\neq 0$.
 If $\F(\Delta)=\{F_1,\ldots,F_m\}$, then $J_k(\Delta)=P_{F_1}^k\sect \cdots \sect P_{F_m}^k$ \cite[Lemma 4.1]{HHT}. In particular, for the facet ideal of $\Delta$, i.e.
 $I(\Delta)= (x_{F_1},\ldots,x_{F_m})$,  one has $J_k(\Delta)=(I(\Delta)^\vee)^{(k)}$ where $(I(\Delta)^\vee)^{(k)}$ is the $k$-th symbolic power of $I(\Delta)^\vee$.

\medskip
Let $B(\Delta)$ be the $S$-subalgebra of $S[t]$ generated by the elements $x^\cb t^k$ where $\cb$ is a squarefree $k$-cover. The algebra $B(\Delta)$ is called the {\em squarefree vertex cover algebra} of $\Delta$. Observe that $B(\Delta)$ is a graded $S$-algebra,
\[
B(\Delta)=\Dirsum_{k\geq 0}L_k(\Delta)t^k.
\]
It is clear that each $L_k(\Delta)$ is a monomial ideal in $S$ and that $L_k(\Delta)\subseteq J_k(\Delta)$ for all $k$.

For a monomial ideal $I\subseteq S$, we denote by $I^{sq}$ the squarefree monomial ideal generated by all squarefree monomials $u\in I$.
The $k$-th squarefree power of a monomial ideal $I$, denoted by $I^{\langle k\rangle}$,  is defined to be $(I^k)^{sq}$.

\begin{Proposition}
\label{alsoeasy}
Let $\Delta$  be a simplicial complex with $\F(\Delta)=\{F_1,\ldots,F_m\}$. Then
$$ L_k(\Delta)^{sq}= \bigcap _{i=1}^m {P_{F_i}}^{\langle k\rangle},$$
and the algebra $B(\Delta)$ is standard graded if and only if $\bigcap _{i=1}^m {P_{F_i}}^{\langle k\rangle}={(\bigcap _{i=1}^m P_{F_i})}^{\langle k\rangle}.$
\end{Proposition}
\begin{proof}
Let $x^{\cb}$ be a squarefree monomial in $L_k(\Delta)^{sq}$. Then  $\sum_{j\in F_i}c_j\geq k$ for all $i$, and  $c_j\leq 1$ for all $j$. So if $u_i=\prod_{j\in F_i}x_j^{c_j}$, then $u_i$ is of degree at least $k$ which implies $u_i\in P_{F_i}^{\langle k\rangle}$. Furthermore,  we have $u_i|x^c$ for all $i$.   Therefore, $x^{\cb}\in \bigcap _{i=1}^m {P_{F_i}}^{\langle k\rangle}$.  On the other hand,  if
 $x^{\cb} \in \bigcap _{i=1}^m {P_{F_i}}^{\langle k\rangle}$, then $\sum_{j\in F_i}c_j\geq k$ for all $i$, or in other words $\cb$ is a $k$-cover. Hence $x^{\cb}\in L_k(\Delta)^{sq}.$

 One has ${(\bigcap _{i=1}^m P_{F_i})}^{\langle k\rangle} \subseteq \bigcap _{i=1}^m {P_{F_i}}^{\langle k\rangle}$. Now the graded algebra $B(\Delta)$ is standard graded if and only if every squarefree $k$-cover of $\Delta$ can be written as a sum of $k$ 1-cover of $\Delta$, and this is the case  if and only if
 $\bigcap _{i=1}^m {P_{F_i}}^{\langle k\rangle} \subseteq {(\bigcap _{i=1}^m P_{F_i})}^{\langle k\rangle}$.
\end{proof}

It turns out that every squarefree $k$-cover of $\Delta$ can be considered as a $1$-cover of a suitable skeleton of $\Delta$ as we see in the next result.
\begin{Theorem}
\label{dual}
Let $\Delta$ be a  simplicial complex of dimension $d-1$ on the vertex set  $[n]$,   and $k \in \{1, \dots, d\}$. Then
\[
L_k(\Delta)^{sq}\subseteq  I(\Delta^{(d-k)})^\vee \quad \text{for all} \quad k.
\]
Furthermore, the following conditions are equivalent:
\begin{enumerate}
\item[(i)]$\Delta$ is a pure simplicial complex;
\item[(ii)]$L_k(\Delta)^{sq}=  I(\Delta^{(d-k)})^\vee \quad \text{for some} \quad k \neq 1$;
\item[(iii)]$L_k(\Delta)^{sq}=  I(\Delta^{(d-k)})^\vee \quad \text{for all} \quad k$.
\end{enumerate}
\end{Theorem}

\begin{proof}
Let  $x^{\cb}$ belong to the minimal set of monomial generators of the ideal $L_k(\Delta)^{sq}$. So $\cb$ is a squarefree $k$-cover of $\Delta$.  Let $C=\supp(\cb) =\{i : c_i=1\}$. Then $C$ contains at least $k$ elements of each facet. Therefore, $C$ also meets all  $(d-k)$-faces of $\Delta$. This implies that $x^\cb \in I(\Delta^{(d-k)})^\vee$.

Now we show that statements (i), (ii) and (iii) are equivalent.

(i)$\Rightarrow$ (iii): Let $x^{\cb}$  be in the minimal set of monomial generators of the ideal
$I(\Delta^{(d-k)})^\vee$. Then $C=\supp(\cb)=\{i : c_i=1\}$ is a minimal vertex cover of $\Delta^{(d-k)}$. Every facet of $\Delta$ is of dimension of $d-1$. Hence the set $C$ contains  at least $k$ vertices of each facet because otherwise there exists a $(d-k)$-face of $\Delta$ which does not intersect $C$. Therefore $\cb$ is a $k$-cover of $\Delta$. Since  $c_i \leq 1$ for all $i$, the cover  $x^\cb$ is  squarefree, and hence   $x^\cb \in L_k(\Delta)^{sq}$.

(iii)$\Rightarrow $(ii):  This implication is trivial.

(ii)$\Rightarrow$ (i): Suppose for some fixed $k \neq 1$, we have $L_k(\Delta)^{sq}=  I(\Delta^{(d-k)})^\vee$. By contrary assume there is a facet $H$ of $\Delta$ of dimension $\ell-1$ where $\ell<d$. Every facet of $\Delta^{(d-k)}$ which is not a subset  of $H$ contains at least one vertex which does not belong to $H$. Hence there exists  a set $A\subseteq [n]$ such that $A \cap H = \emptyset$  and $A \cap F \neq \emptyset$  for all facets $F$ of $\Delta^{(d-k)}$ which are not subsets  of $H$.

First suppose that $\ell-1<d-k$, then $H$ is also a facet of  $\Delta^{(d-k)}$. By choosing a vertex $i$ of $H$, the set $C = A \cup \{ i\}$ is a vertex cover of $\Delta^{(d-k)}$ which contains exactly one vertex of $H$. Hence for the vector $\cb$ with $C=\supp(\cb)$, the monomial $x^{\cb}$ belongs to $I(\Delta^{(d-k)})^\vee$. But since $C$ contains only one vertex of $H$, it does not belong to $L_k(\Delta)^{sq}$, a contradiction.

Next  suppose that $\ell-1 \geq d-k$. By choosing $\ell+k-d$ vertices $i_1,\ldots, i_{\ell+k-d}$ of $H$, the set
 $C= A \cup \{i_1,\ldots, i_{\ell+k-d}\}$ is a vertex cover of $\Delta^{(d-k)}$. Indeed, let $F$ be a facet of $\Delta^{(d-k)}$. If $F\not\subseteq H$, then $F\sect C\neq\emptyset$ because $F\sect A\neq\emptyset$ by the choice of $A$. If $F\subseteq H$, then $|F|=d-k+1$. So
 $F\sect H\neq \emptyset$ because $$|F|+|C\sect H|=(d-k+1)+(\ell+k-d)>\ell=|H|.$$ Hence for the vector $\cb$ with $C=\supp(\cb)$, we have   $x^{\cb} \in I(\Delta^{(d-k)})^\vee$. Thus $x^{\cb}\in  L_k(\Delta)^{sq}$ which implies that $C$ contains at least $k$ elements of $H$.  Furthermore, we know  that   $C$  contains exactly $\ell+k-d$ vertices of $H$. Hence $\ell+k-d \geq k$ which means $\ell \geq d$, a contradiction.
\end{proof}

As an immediate consequence of Theorem~\ref{dual} we obtain

\begin{Corollary}
\label{duality}
Let $\Delta$ be a pure simplicial complex of dimension $d-1$. Then
\[
 I(\Delta^{(k)})^\vee=(I(\Delta)^\vee)^{\langle d-k\rangle} \quad\text{for all} \quad k,
\]
 if and only if $B(\Delta)$ is standard graded.
\end{Corollary}
\begin{proof}
It is enough to notice that $B(\Delta)$ is standard graded if and only if $L_k(\Delta)^{sq}=(I(\Delta)^\vee)^{\langle k\rangle}$ for all $k$. Now Theorem~\ref{dual} yields the result.
\end{proof}

\begin{Corollary}[Duality]
Let $\Delta$ be a pure simplicial complex of dimension $d-1$. Then
\[
 L_j(\Delta^{(d-i)})^{sq}=L_i(\Delta^{(d-j)})^{sq},
\]
where $1\leq i,j \leq d.$
\end{Corollary}
\begin{proof}
Consider integer numbers $i,j\in \{1,\ldots,d\}$. Since $\Delta$ is pure, the simplicial complexes $\Delta^{(d-i)}$ and $\Delta^{(d-j)}$ are  also pure.  So by Theorem~\ref{dual}, we have
$$L_j(\Delta^{(d-i)})^{sq}=  I(\Delta^{((d-i+1)-j)})^\vee,$$
and
$$L_i(\Delta^{(d-j)})^{sq}=  I(\Delta^{((d-j+1)-i)})^\vee.$$
 Thus $L_j(\Delta^{(d-i)})^{sq}=L_i(\Delta^{(d-j)})^{sq}$.
\end{proof}

As an application we consider the following class of ideals introduced in \cite{VHF}. Let $P=\{p_1,\ldots,p_m\}$ be a finite poset  and $r\geq 1$ an integer. We consider  the $(r\times m)$-matrix $X=(x_{ij})$ of indeterminates, and define the ideal $I_r(P)$ generated by all monomials $x_{1j_1}{x_{2j_2}}\cdots  x_{rj_r}$ with $p_{j_1}\leq p_{j_2}\leq\ldots \leq  p_{j_r}$. Let $\Delta_{r}(P)$ be the simplicial complex on the vertex set $V=\{(i,j)\: 1\leq i \leq r,\;1\leq j \leq m  \}$ with the property that $I(\Delta_{r}(P))=I_r(P)$. So $\Delta_r(P)$ is of dimension $r-1$. In \cite[Theorem 1.1]{VHF},  it is shown that
\[
 I(\Delta_{r}(P)^{(k)})^\vee=(I(\Delta_{r}(P))^\vee)^{\langle r-k\rangle} \quad\text{for all} \quad k.
\]
Hence Corollary \ref{duality} implies that $B(\Delta_{r}(P))$ is standard graded. One even has

\begin{Theorem}
\label{evenmaybe}
The $S$-algebra $A(\Delta_{r}(P))$ is standard graded.
\end{Theorem}
\begin{proof}
Let  the  $(r\times m)$-matrix $\cb=[c_{ij}]$ be a $k$-cover of $\Delta_r(P)$ with $k\geq 2$. We show that there is a decomposition of $\cb$ into a 1-cover $\ab$ and a $(k-1)$-cover $\bb$ of $\Delta_r(P)$. This will imply that $A(\Delta_{r}(P))$ is standard graded.

We consider the subset $A$ of $V$  containing the  vertices $(i,j)\in V$ with the following properties:
\begin{enumerate}
\item[(i)] $c_{ij}\neq 0$;
\item[(ii)] There exists a chain $ p_{j_1}\leq p_{j_2}\leq\ldots \leq  p_{j_{i-1}}$ of elements of $P$ with $ p_{j_{i-1}} \leq p_j$ and $c_{t,j_t}=0$, for $t=1,\dots,i-1$.
\end{enumerate}
One should notice that $A$ includes the set $\{(1,j)\in V \: c_{1,j}\neq 0 \}$. First, we show that $A$ is a vertex cover of $\Delta_r(P)$. Suppose $F$ is a facet of $\Delta_r(P)$. So there exists the chain $p_{j_1}\leq p_{j_2}\leq\ldots \leq  p_{j_r}$ of elements of $P$ such that $$F=\{(1,j_1),(2,j_2),\ldots,(r,j_r)\}.$$
Since $\cb$ is a cover of $\Delta_r(P)$ of positive order, the set $D=\{(i,j_i)\in F \: c_{ij_i}\neq 0\}$ is  nonempty. Suppose $t=\min\{s\:(s,j_s)\in D \}$. Then $(t,j_t)$ satisfies the property (i) of elements of $A$. Considering the chain $p_{j_1}\leq p_{j_2}\leq\ldots \leq  p_{j_{t-1}}$, one can see $(t,j_t)$ also satisfies the property (ii) of elements of $A$. Hence $(t,j_t)\in A$. This shows that $A$ is a vertex cover of $\Delta_r(P)$. Let the $(r \times m)$-matrix $[a_{ij}]$ be the squarefree 1-cover of $\Delta_r(P)$ corresponding to $A$. In other words, $[a_{ij}]$  is the $(0,1)$-matrix with $a_{ij}=1$  if and only if $(i,j)\in A$.

Next we show that for every chain $p_{j_1}\leq p_{j_2}\leq\ldots \leq  p_{j_r}$ of elements of $P$ with  $a_{tj_t}>0$, we have $\sum_{s=t}^r c_{s,j_s}\geq k$. Indeed, since $(t,j_t)\in A$,  property (ii) implies that  there exists a chain $ p_{i_1}\leq p_{i_2}\leq\ldots \leq  p_{i_{t-1}}$ of elements of $P$ with $ p_{i_{t-1}} \leq p_{j_t}$ and $c_{s,i_s}=0$, for $s=1,\dots,t-1$. Consider the facet $F$  of $\Delta_r(P)$ corresponding to  the chain $$p_{i_1}\leq p_{i_2}\leq\ldots \leq  p_{i_{t-1}}\leq p_{j_t}\leq\ldots \leq p_{j_r}.$$
Since $\cb$ is a $k$-cover of $\Delta_r(P)$, one has $\sum_{(i,j)\in F}c_{ij}\geq k$. This implies that
$$\sum_{s=1}^{t-1} c_{s,i_s}+\sum_{s=t}^r c_{s,j_s}=\sum_{s=t}^r c_{s,j_s}\geq k.$$

\medskip
Finally, we show that the $(r \times m)$-matrix $[b_{ij}]=[c_{ij}-a_{ij}]$ is a  $(k-1)$-cover of $\Delta_r(P)$. To see this, let $F$ be a facet of $\Delta_r(P)$ corresponding to a  chain $p_{j_1}\leq p_{j_2}\leq\ldots \leq  p_{j_r}$ of elements of $P$. Since $A$ is a vertex cover of $\Delta_r(P)$, the set $A\cap F$ is nonempty. Suppose $A\cap F=\{(i_1,j_{i_1}),(i_2,j_{i_2}),\ldots,(i_t,j_{i_t})\}$ with $i_1< i_2<\ldots< i_t$. By the above discussion, since $a_{i_t j_{i_t}}>0$, we have $\sum_{s=i_t}^r c_{s,j_s}\geq k$. Furthermore, the elements $a_{i_1 j_{i_1}},a_{i_1,j_{i_1}},\ldots,a_{i_{t-1},j_{i_{t-1}}}$ are nonzero because $$\{(i_1,j_{i_1}),(i_2,j_{i_2}),\ldots,(i_{t-1},j_{i_{t-1}})\}\subseteq A.$$
 This implies $c_{i_1 j_{i_1}},c_{i_2,j_{i_2}},\ldots,c_{i_{t-1},j_{i_{t-1}}}$ are nonzero. Hence  $$\sum_{(i,j)\in F}c_{ij}\geq \sum _{s=1}^{t-1} c_{i_s,j_{i_s}} +\sum_{s=i_t}^r c_{s,j_s}\geq (t-1)+\sum_{s=i_t}^r c_{s,j_s}\geq (t-1)+k.$$
Consequently, $$\sum_{(i,j)\in F}b_{ij}=\sum_{(i,j)\in F}c_{ij}-\sum_{(i,j)\in F}a_{ij} \geq ((t-1)+k)-t=k-1.$$ It shows that $\bb$ is a $(k-1)$-cover of $\Delta_r(P)$, and $\cb=\ab+\bb$ is the desired decomposition of $\cb$.
\end{proof}

\section{Comparison of $A(\Delta)$ and $B(\Delta)$}

In view of Corollary~\ref{duality} it is of interest to know when $B(\Delta)$ is standard graded as an $S$-algebra. Of course this is the case if $A(\Delta)$ is standard graded. Thus the following questions arise:
\begin{enumerate}
\item[(1)] Is $A(\Delta)$ standard graded if and only if $B(\Delta)$ is standard graded?

\item[(2)] When do we have  $A(\Delta)=B(\Delta)$?
\end{enumerate}

In general Question (1) does not have a positive answer. The following example was communicated to us by Villarreal. Let $\Delta$ be the simplicial complex with the following facets:
\[
\{1,2\},  \{3,4\},  \{5,6\},  \{7,8\},  \{1,3,7\},  \{1,4,8\},  \{3,5,7\},  \{4,5,8\},  \{2,3,6,8\},  \{2,4,6,7\}
\]
It can be seen  that $\cb=(1,1,1,1,2,0,1,1)$ is an indecomposable 2-cover of $\Delta$, and hence $A(\Delta)$ is not standard graded. However  using   CoCoA \cite{Cocoa},  one can check that the (finitely many) squarefree covers of $\Delta$ of order $\geq 2$ are all decomposable. Thus $B(\Delta)$ is standard graded.

\medskip
Now we consider some cases where Question (1)   has a positive answer. We begin with a general fact about the generators of $B(\Delta)$.

\begin{Lemma}
\label{lunch}
Let $r=\min\{|F|\:\; F\in \F(\Delta)\}$. Then $B(\Delta)$ is generated in degree $\leq r$,  and $L_r(\Delta)^{sq}=(x_1x_2\cdots x_n)$ if and only if each vertex is contained in a $(r-1)$-facet of $\Delta$.
\end{Lemma}

\begin{proof} Let $\cb$ be a squarefree cover of $\Delta$, and $F\in \F(\Delta)$ for which $|F|=r$. If $C=\supp(\cb)$, then $|C\cap F|\leq r$. Hence the order of $\cb$ is at most $r$. This shows that $B(\Delta)$ is generated in degree $\leq r$. For the other statement we first observe that by the definition of the integer $r$ we have  $(x_1x_2\cdots x_n)\subseteq  L_r(\Delta)^{sq}$.  Therefore  $(x_1x_2\cdots x_n)\neq L_r(\Delta)^{sq}$ if and only if there exists $i\in[n]$ such that $x_1\cdots \hat{x}_i\cdots x_n\in L_r(\Delta)^{sq}$. This is the case if and only  if for each facet $F\in\ \Delta$ with $i\in F$ one has $|F\cap\{1,\ldots,\hat{i},\ldots,n\}|\geq r$ which implies that $|F|>r$.
\end{proof}

\medskip
When $\dim\Delta=1$, one may view $\Delta$ as a finite simple graph on $[n]$. In that case we have

\begin{Proposition}
\label{standardgraph}
Let $G$ be a finite simple graph. Then  $B(G)$ is  standard graded  if and only if  $A(G)$ is  standard graded.
\end{Proposition}

\begin{proof}
We may assume that $G$ has no isolated vertices, because adding or removing an  isolated vertex to $G$ does not change the property of $B(G)$,  respectively of  $A(G)$, to be standard graded.

Assume $B(G)$  is standard graded. Then $\cb=(1,\ldots,1)$ is decomposable, or in other words, there exist vertex covers  $C_1,C_2\subseteq [n]$ such that $C_1\union C_2=[n]$ and $C_1\sect C_2=\emptyset$. This implies that each edge of $G$ has exactly one vertex in $C_1$ and one in $C_2$. Therefore $G$ is bipartite. By Theorem~\cite[Theorem5.1.(b)]{HHT} it follows that $A(G)$ is standard graded.
\end{proof}

Let $G$ be a finite simple graph on $[n]$. The {\em edge ideal} of $G$, denoted by $I(G)$, is the ideal generated by $\{x_ix_j\: \{i,j\}\text{ is an edge of } G\}$. We also denote $I(G)^\vee$ by $J(G)$ which is called the {\em cover ideal} of $G$.

An ideal $I$ is said to be {\em normally torsion free} if all the powers $I^j$ have the same associated prime ideals. The following result is shown in  \cite[Theorem 5.9]{SVV}.
\begin{Theorem}[Simis, Vasconcelos, Villarreal]
\label{Villarreal}
Let G be a graph, and suppose $I(G)$ is its edge ideal. The following conditions are equivalent:
\begin{enumerate}
\item[(i)] $G$ is bipartite;
\item[(ii)] $I(G)$ is normally torsion free.
\end{enumerate}
\end{Theorem}

By this theorem, we have another case for which the question (1) has a positive answer; as shown in the next proposition.
\begin{Proposition}
\label{coveringIdeal}
Let $G$ be a graph, and suppose $\Delta$ is the simplicial complex  with $I(\Delta)=J(G)$. Then  $B(\Delta)$ is  standard graded  if and only if $A(\Delta)$ is  standard graded.
\end{Proposition}
\begin{proof}
Suppose $B(\Delta)$ is   standard graded. We show that $G$ is bipartite. By contrary, assume that $G$ has a cycle $l\:\;i_1,i_2,\ldots,i_r$ of odd length. Every facet $F$ of $\Delta$ is a vertex cover of the graph $G$. In addition, $r$ is an odd number, so we have
\begin{eqnarray}
\label{shokr}
 |F\cap  \{i_1,i_2,\ldots ,i_r\} |\geq (r+1)/2.
\end{eqnarray}
for every facet $F$ of $\Delta$. Indeed, if inequality (\ref{shokr}) does not hold, then there  exists an edge of  the cycle $l$ which does not meet $F$.

 Consider the squarefree vector $\cb=(c_1,\ldots,c_n)$ with $c_j=1$ if and only if $j\in C=\{i_1,i_2,\ldots,i_r\}$. Inequality (\ref{shokr}) shows that $\cb$ is a $(r+1)/2$-cover of $\Delta$. Thus the set $C$ is a disjoint union of $(r+1)/2$ vertex covers of $\Delta$ because $B(\Delta)$ is   standard graded. Observe that the  minimal vertex covers of $\Delta$ are exactly the edges of the graph $G$. So the above argument implies that the set of edges of the cycle $l$ contains $(r+1)/2$ pairwise disjoint edges, a contradiction. Hence $G$ is bipartite. So by Theorem~\ref{Villarreal}, the ideal $I(G)$ is normally torsion free. It is well known that in this case the simplicial complex with facet ideal $I(G)^\vee=J(G)$, i.e. $\Delta$, has the standard graded vertex cover algebra; see \cite[Corollary 3.14]{GVV}, \cite[Corollary 1.5 and 1.6]{XIN}    and \cite{HSV}.
\end{proof}

Next we discuss another result regarding question (1).

\medskip
Let $\Delta$ be a simplicial complex. A {\em subcomplex} $\Gamma$ of $\Delta$, denoted by $\Gamma \subseteq \Delta$, is a simplicial complex such that
$\F(\Gamma) \subseteq \F(\Delta)$.
A {\em cycle} of length $r$ of $\Delta$ is a sequence $i_1,F_1,i_2,\ldots , F_r,i_{r+1}=i_1$ where $F_j\in \F(\Delta)$, $i_j \in [n]$ and $v_j,v_{j+1}\in F_j$ for $j=1,\ldots , j=r$.
A cycle is called {\em special} if each facet of the cycle contains exactly two vertices of the cycle.

\begin{Theorem}
\label{no-odd}
Let $\Delta$ be a simplicial complex. Then the following conditions are equivalent:
\begin{enumerate}
\item[(i)] The vertex cover algebra $B(\Gamma)$ is standard graded for all $\Gamma \subseteq \Delta;$
\item[(ii)] The vertex cover algebra $A(\Gamma)$ is standard graded for all $\Gamma \subseteq \Delta;$
\item[(iii)] $\Delta$ has no special odd cycles.
\end{enumerate}
\end{Theorem}
\begin{proof}
The equivalence of (ii) and (iii) is known by \cite[Theorem 2.2]{XIN}; see also \cite[Proposition 4.10]{GRV}. We show (i) and (iii) are also equivalent.

(i)$\Rightarrow$ (iii) Following the proof of \cite[Lemma 2.1]{XIN}, let $i_1,F_1,i_2,\ldots , F_r,i_{r+1}=i_1$ be a special cycle of $\Delta$. Consider the subcomplex $\Gamma \subseteq \Delta$ with $\F(\Gamma)=\{F_1,\ldots , F_r\}$. By the definition of the special cycles, one has $|F_j\cap \{i_1,\ldots, i_r\}|= 2$ for each $j=1,\ldots,r$. So
$\{i_1,\ldots, i_r\}$ corresponds to a squarefree 2-cover of $\Gamma$, that is, there exists a squarefree 2-cover $\cb$ of $\Gamma$ with $\supp(\cb)=\{i_1,\ldots, i_r\}$. Since by assumption $B(\Gamma)$ is standard graded, there are disjoint vertex covers $C_1$ and $C_2$ of $\Gamma$ such that $\{i_1,\ldots, i_r\}=C_1\cup C_2$. So the sets $F_j\cap C_1$ and $F_j\cap C_2$ are nonempty for all $j$. Furthermore, since  every facet of a special cycle contains exactly two vertices of the cycle,  it follows that  $|F_j\cap C_1|=|F_j\cap C_2|=1$. Hence $C_1$ and $C_2$ have the same number of vertices which implies that $r$ is an even number.

(iii)$\Rightarrow$ (i) By \cite[Theorem 2.2]{XIN}, when $\Delta$ has no special odd cycle, statement (ii) holds. Therefore  $B(\Gamma)$ is standard graded for all $\Gamma \subseteq \Delta$.
\end{proof}

Next we discuss question (2) in which we ask when $B(\Delta)=A(\Delta)$. In the case that $\dim \Delta=1$, in other words if $\Delta$ can be identified with a graph $G$, we have a complete answer which is an immediate consequence of the following result, see
\cite[Proposition 5.3]{HHT}:
\begin{Proposition}
\label{quoted}
Let G be a simple graph on $[n]$. Then the following conditions are equivalent:

\begin{enumerate}
\item[(i)] The graded $S$-algebra $A(G)$ is generated by $x_1x_2\cdots x_nt^2$ together with those monomials $x^{\cb}t$ where $c$ is a squarefree 1-cover of $G$;
\item[(ii)] For every cycle $C$ of $G$ of odd length and for every vertex $i$ of $G$ there exist a vertex $j$ of the cycle $C$ such that $\{i,j\}$ is an edge of $G$.
\end{enumerate}
\end{Proposition}

By using this result we get

\begin{Proposition}
\label{graphequality}
Let G be a simple graph on $[n]$. Then the following conditions are equivalent:
\begin{enumerate}
\item[(i)] $A(G)=B(G);$
\item[(ii)] For every cycle $C$ of $G$ of odd length and for every vertex $i$ of $G$ there exist a vertex $j$ of the cycle $C$ such that $\{i,j\}$ is an edge of $G$.
\end{enumerate}
\end{Proposition}
\begin{proof}
(i)$\Rightarrow$ (ii): We may assume that $G$ has no isolated vertex. So by Lemma~\ref{lunch}, the graded $S$-algebra $B(G)$  is generated by $x_1x_2\cdots x_nt^2$ together with those monomials $x^\cb t$ where $\cb$ is a squarefree $1$-cover of $G$. Since we assume that $A(G)=B(G)$, the same holds true for $A(G)$. Hence (ii)  follows from Proposition~\ref{quoted}.

(ii)$\Rightarrow$ (i): If (ii) holds, then Proposition~\ref{quoted} implies that  $A(G)$ is  generated by the monomials corresponding to squarefree covers.  So $A(G)=B(G)$.
\end{proof}

 Let $\Delta$ be a simplicial complex on $[n]$. For a subset $W$ of $[n]$, we define the {\em restriction} of $\Delta$ with respect to  $W$, denoted by $\Delta_W$, to be  the subcomplex of $\Delta$ with
 $$\F(\Delta_W) = \{F\in \F(\Delta) \: F \subseteq W\}.$$

\begin{Lemma}
\label{restriction}
Let $\Delta$ be a simplicial complex on $[n]$ and $W\subseteq [n]$. If $B(\Delta)=A(\Delta)$, then $B(\Delta_W)=A(\Delta_W)$.
\end{Lemma}
\begin{proof} We assume that there exists at least one facet $F$ of $\Delta$ such that $F \subseteq W$, otherwise there is nothing to prove. Furthermore, without loss of generality we assume that $W=\{1,\ldots,t\}$.  Let $\mathbf{c}'=(c_1,\ldots,c_t)$ be an indecomposable $k$-cover of  $\Delta_W$ with $k>0$. We will show that $\cb'$ is squarefree.
We  extend $\cb'$ to the $k$-cover $\cb=(c_1,\ldots,c_t,k,\ldots,k)$ of $\Delta$. Since $B(\Delta)=A(\Delta)$, there exists a decomposition of $\cb=\ab+\bb$ where $\ab$ is an indecomposable squarefree $i$-cover of $\Delta$ with $i>0$, $\bb$ is a $j$-cover of $\Delta$, and $i+j=k$.  Let  $\ab'$ and $\bb'$ be respectively the restrictions of $\ab$ and $\bb$ to the first $t$-components. For every facet $F\in \F(\Delta_W)$ we have $\sum_{\ell\in F}a_{\ell}\geq i>0$ This implies that $\ab'\neq 0$ . Suppose $\bb'\neq 0$, then $\cb'=\ab'+\bb'$ is a decomposition of $\cb'$, a contradiction. Hence $\bb'=0$, and so $\cb'=\ab'$ which means that $\cb'$ is squarefree.
\end{proof}

 By Theorem~\ref{no-odd}, so far we know that for a simplicial complex $\Delta$ without any special odd cycle we have $B(\Delta)=A(\Delta)$.  On the other hand, if $\Delta$ contains special odd cycles, then $A(\Delta)$ and $B(\Delta)$  may not be  equal. This may even happen if the facets of $\Delta$ are precisely the facets of a special odd cycle, as the following two examples demonstrate.  Figure~\ref{cycle5} shows a simplicial complex $\Delta_1$ of dimension 2 such that $$1,F_1,2,F_2,3,F_3,4,F_5,5,F_5,1$$ is a special odd cycle of length 5.
\begin{figure}
\begin{center}
\begin{pspicture}(0,-0.6)(5,3.6)
\psdots(0,0)(2,0)(4,0)(1,1.6)(3,1.6)(2,3.2)(5.25,2.8)
\pspolygon[fillcolor=light,fillstyle=solid](0,0)(1,1.6)(2,0)
\pspolygon[fillcolor=light,fillstyle=solid](2,0)(3,1.6)(4,0)
\pspolygon[fillcolor=light,fillstyle=solid](1,1.6)(2,3.2)(3,1.6)
\psline(2,3.2)(5.25,2.8)(4,0)
  \rput(2,3.6){2}
  \rput(0.6,1.6){1}
  \rput(3.4,1.6){7}
  \rput(0,-0.4){6}
  \rput(2,-0.4){5}
  \rput(4,-0.4){4}
  \rput(5.5,2.8){3}
  \rput(2,2.3){$F_1$}
  \rput(3,0.7){$F_4$}
  \rput(1,0.7){$F_5$}
  \rput(3.7,3.3){$F_2$}
  \rput(5,1.4){$F_3$}
\end{pspicture}
\end{center}
\caption{}\label{cycle5}
\end{figure}
One can see the vector $(1,0,2,0,1,0,1)$ is an indecomposable 2-cover of $\Delta_1$. Therefore $B(\Delta_1)\neq A(\Delta_1)$.

Next consider the simplicial complex $\Delta_2$ as shown in Figure~\ref{cycle3}. In this case the facets of $\Delta_2$ form the  special odd cycle $6,F_1,2,F_2,4,F_3,6$ of length 3 and the equality $B(\Delta_2)=A(\Delta_2)$ holds. Indeed, the equality $B(\Delta_2)=A(\Delta_2)$ is a consequence of a more general result given in  the next theorem.
\begin{figure}
\begin{center}
\begin{pspicture}(0,-0.6)(4,3.6)
\psdots(0,0)(2,0)(4,0)(1,1.6)(3,1.6)(2,3.2)
\pspolygon[fillcolor=light,fillstyle=solid](0,0)(1,1.6)(2,0)
\pspolygon[fillcolor=light,fillstyle=solid](2,0)(3,1.6)(4,0)
\pspolygon[fillcolor=light,fillstyle=solid](1,1.6)(2,3.2)(3,1.6)
  \rput(2,3.6){1}
  \rput(0.6,1.6){6}
  \rput(3.4,1.6){2}
  \rput(0,-0.4){5}
  \rput(2,-0.4){4}
  \rput(4,-0.4){3}
  \rput(2,2.3){$F_1$}
  \rput(3,0.7){$F_2$}
  \rput(1,0.7){$F_3$}
\end{pspicture}
\end{center}
\caption{}\label{cycle3}
\end{figure}

\medskip
These two examples  show that it is not easy to classify simplicial complexes $\Delta$  containing odd cycles for  which $B(\Delta)=A(\Delta)$. Therefore, we restrict ourselves to consider special classes of simplicial complexes. Firstly, we consider simplicial complexes  satisfying the following intersection property: let $\Delta$ be a simplicial complex with  $\mathcal{F}(\Delta)=\{F_1,\ldots,F_m\}$. We say that $\Delta$ has the {\em strict intersection property}  if
\begin{enumerate}
\item[ ($\text{I}_1$)] $|F_i\sect F_j|\leq 1$ for all $i\neq j$;
\item[ ($\text{I}_2$)] $F_i\sect F_j\sect F_k=\emptyset$ for pairwise distinct $i$,$j$ and $k$.
\end{enumerate}

Given a simplicial complex  $\Delta$ with $\mathcal{F}(\Delta)=\{F_1,\ldots,F_m\}$ satisfying the strict intersection property, we define the {\em intersection graph $G_\Delta$} of $\Delta$ as follows: $$V(G_\Delta)=\{v_1,\ldots,v_m\}$$ is the vertex set of $G_\Delta$, and $$E(G_\Delta)=\{ \{v_i,v_j\}\:\; i\neq j \quad \text{and}\quad F_i\sect F_j\neq \emptyset\}$$
is the edge set of $G_\Delta$.

Note that if  $W$ is a subset   of $[n]$,  then $\Delta_W$ satisfies again the strict intersection property and the  graph $G_{\Delta_W}$ is the  subgraph of $G_\Delta$ induced by
$$S=\{v_i\in V(G_{\Delta})\: F_i\subseteq W \}.$$

\begin{Lemma}
\label{odd-cycle}
Let $\Delta$ be a simplicial complex satisfying the strict intersection property. If $G_\Delta$ is an odd cycle, then $A(\Delta)$ is minimally generated in degree $1$ and $2$,  and   $B(\Delta)=A(\Delta)$.
\end{Lemma}
\begin{proof}
Let $[n]$ be the vertex set of $\Delta$ and $\F(\Delta)=\{F_1,\ldots,F_m\}$. Since $G_\Delta$ is an odd cycle and $\Delta$ has the strict intersection property, we may assume $$i_1, F_1,i_2,F_2,\ldots,i_m,F_m, i_{m+1}=i_1$$   is the special odd cycle corresponding to $G_\Delta$ where $\{i_j\}=F_{j-1}\cap F_j$ for $j=2,\ldots,m$ and $\{i_1\}=F_{1}\cap F_m$. We consider a non-squarefree $k$-cover $\cb$ of $\Delta$  with $k>0$ and show that it has a decomposition $\ab+\bb$ such that $\ab$ is a  squarefree cover of positive order. This will imply that $B(\Delta)=A(\Delta)$. If all  entries $c_{i_1},\ldots,c_{i_m}$   of the $k$-cover $\cb$  are nonzero,  then we have the decomposition $\cb=\ab+\bb$ where $\ab$ is the squarefree 2-cover of $\Delta$ with $a_{i_1}=\ldots=a_{i_m}=1$ and all  other entries of $\ab$ are zero, and $\bb$ is the $(k-2)$-cover $\cb-\ab$. So  we may assume that at least one of the entries $c_{i_1},\ldots,c_{i_m}$, say  $c_{i_1}$, is zero. Let $\Gamma$ be the simplicial complex on the vertex set $[n]\setminus \{i_1\}$ with
$$\F(\Gamma)= \{F\setminus \{i_1\}\:F\in \F(\Delta)\}.$$
Consider the vector $\cb'=(c_1,\ldots, \hat{c}_{i_1},\ldots,c_m)$. The vector $\cb'$ is also a $k$-cover of $\Gamma$ because $c_{i_1}= 0$. Since the simplicial complex $\Gamma$ has no special cycle, by Theorem~\ref{no-odd}, the $S$-algebra $A(\Gamma)$ is standard graded. Therefore, there exists a decomposition $\ab'+\bb'$ of $\cb'$ such that $\ab'$ is a squarefree 1-cover and $\bb'$ is a $(k-1)$-cover of $\Gamma$. The vector $\ab$ with  $a_{i_1}=0$ and $a_j=a'_j$ for $j\neq i_1$, is a 1-cover of $\Delta$ and the vector $\bb$ with $b_{i_1}=0$ and $b_j=b'_j$ for $j\neq i_1$, is a $(k-1)$-cover of $\Delta$. Hence $\ab+\bb$ gives us the desired decomposition of $\cb$.

It follows from the above discussion that $A(\Delta)$ is  generated by $x_{i_1}x_{i_2}\cdots x_{i_m}t^2$ and the monomials $x^{\cb}t$ where $\cb$ is an indecomposable squarefree $1$-cover of $\Delta$. This is a minimal set of generators of $A(\Delta)$ because the 2-cover $(1,\ldots,1)$ is not decomposable. In fact since $m$ is an odd number, if $(1,\ldots,1)=\ab+\bb$, then one of $\ab$ and $\bb$ has at most $(m-1)/2$ nonzero entries which means that it is not a 1-cover.
\end{proof}

\begin{Theorem}
\label{str-intersec-prop}
Let $\Delta$ be a simplicial complex satisfying the strict intersection property and suppose that no two cycles of $G_\Delta$   have precisely two edges in common.  Then $B(\Delta)=A(\Delta)$ if and only if each connected component of $G_\Delta$ is a bipartite graph or an odd cycle.
\end{Theorem}

\begin{proof}
Let $G_1,\dots,G_t$ be the connected components of $G_\Delta$, $\Delta_1,\ldots,\Delta_t$ be the corresponding connected components of $\Delta$ and $\{v_{j1},\ldots,v_{js_j}\}$ be the vertex set of $G_j$. Then a $k$-cover $\cb$ of $\Delta$ can be decomposed into an $i$-cover $\ab$ and a $j$-cover $\bb$ of $\Delta$ if and only if the $k$-cover $\cb_j=(c_{j1},\ldots,c_{js_j})$ of $\Delta_j$ can be decomposed to $i$-cover $(a_{j1},\ldots,a_{js_j})$ and $j$-cover $(b_{j1},\ldots,b_{js_j})$ of $\Delta_j$ for all $j$. Hence $B(\Delta)=A(\Delta)$ if and only if $B(\Delta_j)=A(\Delta_j)$ for all $j$. Therefore, it is  enough to consider the case that $G_\Delta$ is connected.

 First, let $B(\Delta)=A(\Delta)$.  We assume that  $G_\Delta$ is not bipartite and show that it is an odd cycle. Let $[n]$ be the set of vertices of $\Delta$ and $\F(\Delta)=\{F_1,\ldots,F_m\}$. Since  $G_\Delta$ is not bipartite, it has an odd cycle $C$. We may assume that the cycle $C\: v_{i_1},\ldots,v_{i_t}$ has no chord because if  $\{v_{i_r},v_{i_s}\}$  is a chord of $C$, then either $v_{i_1},\ldots,v_{i_r},v_{i_s},\ldots,v_{i_t}$ or $v_{i_r},v_{i_{r+1}},\ldots,v_{i_s}$ is an odd cycle. Suppose  the special cycle  corresponding to $C$ in $\Delta$, after a relabeling of facets, is  the cycle
$$\C:\; i_1,F_1,i_2,\ldots,i_r,F_r,i_{r+1}=i_1,$$  where $\{i_j\}=F_{j-1}\cap F_j$ and $r$ is an odd integer. We show that $m=r$. This will imply that  $G_\Delta =C$ because  $C$ has no chord.

Assume $m>r$. We consider two cases and in each case we will find a non-squarefree indecomposable cover of $\Delta$, contradicting our assumption that $B(\Delta)=A(\Delta)$.

{\em Case1.} Suppose each  vertex of $\Delta$ belongs to  at least one of the facets of the cycle $\C$. It is enough to consider the case that $m=r+1$. Indeed, if $m>r+1$, we set $W=[n]\setminus \Union_{j=r+2}^mF_j$ and consider $\Delta_W$.  The strict intersection property implies that $\Delta_W$  has the special odd cycle
 $$\C:\; i_1,F_1',i_2,\ldots,i_r,F_r',i_{r+1}=i_1,$$
 where $F_i' =F_i\setminus \Union_{j=r+2}^mF_j$ for all $i$. Moreover, $\F(\Delta_W)=\{F_1',\ldots, F_r',F_{r+1}\}$.  Lemma~\ref{restriction} implies that $B(\Delta)\neq A(\Delta)$ if $B(\Delta_W)\neq A(\Delta_W)$.

There exist two facets   of $\C$,  say $F_1$ and $F_s$, which intersect $F_{r+1}$.  Since $r$ is an odd integer, one of the cycles
$$ j_1,F_1,i_2,F_2,\ldots,i_s,F_s,j_s,F_{r+1},j_1$$ or $$j_1, F_{r+1},j_s,F_s,i_{s+1},\ldots,F_r, i_1,F_1,j_1$$ is of odd length where $\{j_1\}=F_{r+1}\cap F_1$ and
$\{j_s\}=F_{r+1}\cap F_s$. Without loss of generality, we assume $ j_1,F_1,i_2,F_2,\ldots,i_s,F_s,j_s,F_{r+1},j_1$ is an odd cycle and call it $\D$. One has $r-s>1$ because  $r-s$ is an odd number and if we had    $r-s=1$,  then the  cycles $v_1,v_{r+1},v_s,v_r$ and $v_1,\ldots,v_r$  in $G_\Delta$ would have exactly two edges in common, contradicting our assumption.

 Now we consider the vector $\cb$ with  $c_{j_1}=c_{j_s}=c_{i_2}=\ldots=c_{i_s}=1$, $c_{i_{s+2}}=\ldots=c_{i_r}=2$ and all the  other entries of $\cb$ are zero. The vector $\cb$ is a non-squarefree 2-cover of $\Delta$ while $x^{\cb}t^2\not\in B(\Delta)$. In fact, if $x^{\cb}t^2\in B(\Delta)$, then there is a decomposition $\cb=\ab+\bb$ by an    indecomposable squarefree cover $\ab$   and a cover $\bb$ such that either $\ab$ is a 2-cover or $\ab$ and $\bb$ are both 1-covers of $\Delta$.  Since $\ab$ is squarefree, one has $\sum_{j\in F_{s+1}}a_j \leq 1$. Hence the order of $\ab$ cannot be 2.  On the other hand, since $\sum_{j \in F_i }a_j \geq 1$ for all facets $F_i$ of $\D$, at least $(s/2)+1$  of the  entries  $a_{j_1},a_{j_s},a_{i_2},\ldots,a_{i_s}$ of $\ab$  are nonzero. Therefore at most $s/2$  of the  entries  $b_{j_1},b_{j_s},b_{i_2},\ldots,b_{i_s}$  of $\bb$ are nonzero. Therefore $\bb$ is not a 1-cover of $\Delta$, and so $x^{\cb}t^2\not\in B(\Delta)$.

 {\em Case 2}. Suppose there exists a  vertex $t$ of $\Delta$ which belongs to none  of the facets of the cycle $\C$.
 We may assume that $t$ is the only vertex of $\Delta$ with this property. Because if $t_1,\ldots,t_s$ are the other vertices of $\Delta$ with the same property, then we  consider the  restriction of $\Delta$ to the set $W= [n]\backslash \{t_1,\ldots,t_s\}$,  and show that  $B(\Delta_W) \neq A(\Delta_W)$. Then by applying Lemma~\ref{restriction} we obtain  $B(\Delta) \neq A(\Delta)$.

We may assume that every facet of $\Delta$ which is not a facet of $\C$  contains $t$. Otherwise we  restrict  $\Delta$ to $W'=[n]\setminus \{t\}$, and by Case 1 it follows that $B(\Delta_{W'})\neq  A(\Delta_{W'})$ which implies that  $B(\Delta)\neq  A(\Delta)$.  Let $\cb$  be the vector with  $c_t=2$, $c_{i_1}=c_{i_2}=\ldots=c_{i_r}=1$ and all other entries of $\cb$ be  zero. The vector $\cb$ is a 2-cover of $\Delta$. Suppose $\cb=\ab+\bb$ for some squarefree indecomposable cover $\ab$ of positive order and some cover $\bb$ of $\Delta$. Then    at least $(r+1)/2$  of the entries $a_{i_1},a_{i_2},\ldots,a_{i_r}$  of $\ab$ are nonzero,  and consequently at most $(r-1)/2$ of  the entries $b_{i_1},b_{i_2},\ldots,b_{i_r}$ of $\bb$ are nonzero. Hence $\bb$ is not a $1$-cover. Furthermore,  if $F$ is a facet of $\Delta$ containing $t$, then $\sum_{j\in F}a_j=a_t=1$. Hence $\ab$ is not 2-cover of $\Delta$. Therefore  $\cb$ is an indecomposable non-squarefree 2-cover of $\Delta$.

 Conversely, we suppose  $G_\Delta$ is a bipartite graph or an odd cycle. If $G_\Delta$ is bipartite,  then $\Delta$ has no special odd cycle. Therefore by Theorem~\ref{no-odd}, $A(\Delta)$ is standard graded and consequently $B(\Delta)=A(\Delta)$. The equality for the  case that $G_\Delta$ is an odd cycle, has been shown in Lemma~\ref{odd-cycle}.
\end{proof}

Consider the simplicial complexes $\Delta_1$ and $\Delta_2$ as shown in Figure~\ref{2common}. They both satisfy the  strict intersection property and have the  same intersection graph. The cycles $v_1,v_2,v_3$ and $v_1,v_2,v_3,v_4$ of this intersection graph have exactly two edges in common. We have $B(\Delta_1)=A(\Delta_1)$ but $B(\Delta_2)\neq A(\Delta_2)$.

In fact,  vector $(1,0,2,0,1,1)$ is an indecomposable 2-cover of $\Delta_2$ which shows  $B(\Delta_2)\neq A(\Delta_2)$. However, $A(\Delta_1)$ is even standard graded. Let $\cb=(c_1,\ldots,c_5)$ be a $k$-cover of $\Delta_1$ with $k\geq 2$. We show that there is a decomposition of $\cb$ to a 1-cover $\ab$ and  $(k-1)$-cover $\bb=\cb -\ab$ of  $\Delta_1$. If $c_1$ and $c_3$ are both nonzero, then $\ab=(1,0,1,0,0)$ gives the desired decomposition because the set $\{1,3\}$ meets each facet of $\Delta_1$ at exactly one vertex. Therefore, we may assume $c_1=0$ or $c_3=0$. By the same argument, one may assume  $c_2=0$ or $c_4=0$. So it is enough to consider the case that  $c_1$ and $c_2$ are zero or the case that $c_1$ and $c_4$ are  zero. However, $\cb$ is a $k$-cover of positive order so the entries $c_1$ and $c_4$ cannot be both zero. Now since the vector $\cb=(0,0,c_3,c_4,c_5)$ is a $k$-cover of $\Delta_1$,  one obtains that  $c_i\geq k$ for $i=3,4,5$.  This implies that the vector $\ab=(0,0,1,1,1)$ and $\bb=\cb-\ab$ give the desired decomposition of $\cb$.

\begin{figure}
\begin{center}
\begin{pspicture}(-0.3,-6.5)(14.3,4.3)
\psdots(0,0)(3,0)(1.5,2)(0,4)(3,4)
\pspolygon[fillcolor=medium,fillstyle=solid](0,0)(3,0)(1.5,2)
\pspolygon[fillcolor=medium,fillstyle=solid](1.5,2)(0,4)(3,4)
\psline(0,0)(0,4)
\psline(3,0)(3,4)
 \rput(1.6,-1.2){${\bf \Delta_1}$}
 \rput(1.5,1){$F_3$}
 \rput(-0.4,2){$F_4$}
 \rput(1.5,3){$F_1$}
 \rput(3.4,2){$F_2$}
 \rput(-0.3,4.3){1}
 \rput(3.3,4.3){2}
 \rput(3.3,-0.3){3}
 \rput(-0.3,-0.3){4}
 \rput(1.1,2){5}
 %%%%%%%%%%%%%%%%
 \psdots(9,0)(12,0)(10.5,2)(9,4)(12,4)(14,2)
\pspolygon[fillcolor=medium,fillstyle=solid](9,0)(12,0)(10.5,2)
\pspolygon[fillcolor=medium,fillstyle=solid](10.5,2)(9,4)(12,4)
\pspolygon[fillcolor=medium,fillstyle=solid](12,4)(12,0)(14,2)
\psline(9,0)(9,4)
 \rput(10.6,-1.2){${\bf \Delta_2}$}
 \rput(10.5,1){$F_3$}
 \rput(8.6,2){$F_4$}
 \rput(10.5,3){$F_1$}
 \rput(13,2){$F_2$}
 \rput(8.7,4.3){1}
 \rput(12.3,4.3){2}
 \rput(14.3,2){3}
 \rput(11.7,-0.3){4}
 \rput(8.7,-0.3){5}
 \rput(10.1,2){6}
%
%Intersection Graph
%
\psdots(4.5,-3.5)(7,-2)(9.5,-3.5)(7,-5)
\psline(4.5,-3.5)(7,-2)(9.5,-3.5)(7,-5)(4.5,-3.5)
\psline(7,-2)(7,-5)
 \rput(7,-1.6){$v_1$}
 \rput(9.9,-3.5){$v_2$}
 \rput(7,-5.4){$v_3$}
 \rput(4.1,-3.5){$v_4$}
 \rput(7,-6){The intersection graph of $\Delta_1$ and $\Delta_2$}
  \end{pspicture}
\end{center}
\caption{}\label{2common}
\end{figure}

\medskip

\section{Vertex covers of principal Borel sets}
The classes of simplicial complexes considered so far for which $B(\Delta)=A(\Delta)$, happened to have the property that  $A(\Delta)$  is generated over $S$  in degree at most 2. In this section we present classes of simplicial complexes $\Delta$ such that $B(\Delta)=A(\Delta)$ and  $A(\Delta)$ has generators in higher degrees.

 We will consider a family of simplicial complexes whose set of facets corresponds to a Borel set. Recall that a subset $\B\in 2^{[n]}$ is called {\em Borel} if
whenever $F\in B$ and $i < j$ for some $i\in [n]\setminus F$ and $j\in F$, then $(F \setminus \{j\}) \cup \{i\}\in \B$. Elements $F_1,\ldots,F_m\in \B$  are called {\em Borel generators} of $\B$, denoted by $\B=B( F_1,\ldots,F_m)$, if $\B$ is the smallest Borel subset of $2^{[n]}$ such that $F_1,\ldots,F_m\in \B$.  A Borel set B is called
{\em principal} if there exists $F\in \B$ such that $\B = B(F)$.

A squarefree monomial ideal $I\subseteq S$ is called a {\em squarefree Borel ideal} if there exists a Borel set $\B\subseteq 2^{[n]}$ such that
\[
I=(\{x_F \:\; F\in \B \}).
\]
If $\B=B( F_1,\ldots,F_m)$, then the monomials $x_{F_1},\ldots,x_{F_m}$ are called the {\em Borel generators} of $I$.
The ideal $I$ is called  a {\em squarefree principal Borel ideal}  if $\B$ is principal Borel.

It is known that the Alexander dual of a squarefree Borel ideal is again squarefree Borel \cite{FMS}. In the case that $I$ is  squarefree principal Borel, the following result is shown in \cite[Theorem 3.18]{FMS}.

\begin{Theorem}[Francisco, Mermin, Schweig]
\label{BorelAlexander}
Let $I$ be a squarefree principal Borel ideal with the Borel generator $x_F$ where $F=\{i_1<i_2<\cdots < i_d\}$. Then the Alexander dual $I^\vee$ of $I$ is the squarefree Borel ideal with the Borel generators $x_{H_1},\ldots,x_{H_d}$ where $H_q=\{q,q+1,\ldots,i_q\}$ for  $q=1,\ldots,d$.
\end{Theorem}
\medskip
Let $G_1=\{i_1<\cdots<i_d\}$ and $G_2=\{j_1<\cdots<j_d\}$ be subsets of $[n]$. It is said that {\em $G_1$ precedes $G_2$ (with respect to the Borel order)}, denoted by $G_1\prec G_2$, if $i_s\leq j_s$ for all $s$. By \cite[Lemma 2.11]{FMS}, $F\in B( F_1,\ldots,F_m) $ if and only if $F$ precedes $F_i$, for some $i=1,\ldots,m.$

\begin{Lemma}
\label{skeleton}
Let $\B=B( F_1,\ldots,F_m)$ be a  Borel set with Borel generators  $F_j=\{i_{j,1}<i_{j,2}<\cdots < i_{j,d_j}\}$ for $j=1,\ldots,m$, and suppose $\Delta$ is the simplicial complex with  $\F(\Delta)=\B$.
Then  for the $q$-skeleton $\Delta^{(q)}$ of $\Delta$, the set $\F(\Delta^{(q)})$ is  a  Borel set with the Borel generators $G_1,\ldots,G_m$ such that $G_j=\{i_{j,d_j-q}<i_{j,d_j-q+1}<\cdots<i_{j,d_j}\}$ if $d_j> q$, and $G_j=F_j$ if $d_j\leq q$.
\end{Lemma}
\begin{proof}
Let $\B\subseteq 2^{[n]}$, and  $G$ be a subset of $[n]$. First assume that $|G|\leq q$. In this case $G$ is a facet of $\Delta^{(q)}$ if and only if $G$ is a facet of $\Delta$, and this is the case if and only if $G$ precedes  $F_j$ for some $j$ with $d_j\leq q$.

Next assume that $|G|=q+1$.  We must show that $G$ is a facet of $\Delta^{(q)}$  if and only if $G$ precedes $\{i_{j,d_j-q},i_{j,d_j-q+1},\ldots,i_{j,d_j}\}$ for some $j$ with $d_j> q$.  The set $G$ is a facet of $\Delta^{(q)}$ if and only if there exists a facet $H$ of $\Delta$, preceding  $F_j$ for some $j$ with $d_j> q$, and $G\subseteq H$.  So for simplicity, we may assume that $\B$ is a principal Borel set with the Borel generator $F=\{i_1<i_2<\cdots < i_d\}$ where $d>q$, and show that $G$ is a facet of $\Delta^{(q)}$ if and only if $G$ precedes $\{i_{d-q},\ldots,i_d\}$. Firstly, let $G= \{k_{d-q}<k_{d-q+1}<\cdots<k_d\}\subseteq [n]$, and suppose $G$ precedes $\{i_{d-q},\ldots,i_d\}$. So we have $k_j\leq i_j$, for $j=d-q,\ldots,d$. Let $r$ be an integer such that the cardinality of $[r]\cup \{k_{d-q},k_{d-q+1},\cdots,k_d\}$ is $d$. The set $H=[r]\cup \{k_{d-q},k_{d-q+1},\cdots,k_d\}$ precedes $F$ and obviously includes $G$. This means that $G$ is a facet of $\Delta^{(q)}$. On the other hand, if there exists a set $H$ such that $G\subseteq H$ and $H$ precedes $F$, then $G$ precedes  $\{i_{d-q},\ldots,i_d\}$.
\end{proof}

By the preceding lemma, we have the following generalization of Theorem \ref{BorelAlexander}(\cite[Theorem 3.18]{FMS}).

\begin{Proposition}
\label{B-generators}
Let $\B=B( F)$ be a principal Borel set with Borel generator  $F=\{i_1<i_2<\cdots < i_d\}$, and let $\Delta$ be the simplicial complex with  $\F(\Delta)=\B$. Then the $S$-algebra $B(\Delta)$ is generated by the elements $x_Ht^k$, for k=1,\ldots,d, where
\[
 H\in B(\{q,q+1,\ldots, i_{k+q-1}\}\:\;    q=1,\ldots,d-k+1).
\]
\end{Proposition}
\begin{proof}
First observe that by Lemma \ref{lunch}, $B(\Delta)$ is generated in degree $\leq d$. So it is enough to show that for each  $k=1,\ldots,d$, the monomials $x_Ht^k$ with
$H\in B(\{q,q+1,\ldots ,i_{k+q-1}\}\:\;   q=1,\ldots,d-k+1)$  generate $L_k(\Delta)^{sq}$. Since $\Delta$ is pure, by Theorem \ref{dual}, this is the case if these monomials generate $I(\Delta^{(d-k)})^\vee$. By Lemma~\ref{skeleton}, the ideal $I(\Delta^{(d-k)})$ is a Borel ideal with the Borel generator $x_{i_k}x_{i_{k+1}}\cdots x_{i_d}$. Hence Theorem~\ref{BorelAlexander} implies the result.
\end{proof}

\begin{Proposition}
\label{borel}
Let $\B=B(F_1,\ldots,F_m)$ be a Borel set such that  $|F_i|=|F_j|$ for all $i,j$, and  suppose $\Delta $ is a simplicial complex with $\F(\Delta)=\B$. Then $L_k(\Delta)^{sq}$ is a squarefree Borel ideal for all $k$.
\end{Proposition}
\begin{proof}
Since $|F_i|=|F_j|$ for all $i,j$, the simplicial complex $\Delta$ is pure. Considering Theorem~\ref{dual}, this implies that $L_k(\Delta)^{sq}=I(\Delta^{(d-k)})^\vee$ for all $k$. By Lemma~\ref{skeleton}, the ideal $I(\Delta^{(d-k)})$ is a squarefree Borel ideal. As it is shown in Theorem~\ref{BorelAlexander}, the Alexander dual of a principal squarefree Borel ideal is squarefree Borel. Since $(I+J)^\vee=I^\vee\cap J^\vee$ for every squarefree monomial ideals $I$ and $J$, by squarefree version of \cite[Proposition 2.16]{FMS} one obtains that  $I(\Delta^{(d-k)})^\vee$ is a squarefree Borel ideal. Hence $L_k(\Delta)^{sq}$ is squarefree Borel.
\end{proof}
\medskip
In Proposition \ref{borel}, suppose that the cardinality of the Borel generators of $\B$ are not  the same.  Consider the simplicial complex $\Delta$  with $\F(\Delta)=\B'$ where $\B'$ is the set of maximal elements of $\B$, with respect to inclusion. Then the ideal $L_k(\Delta)^{sq}$ does not need to be  squarefree Borel. For example, if $\B=B(\{1,4\},\{1,2,3\})$, then $\F(\Delta)=\{\{1,4\},\{1,2,3\}\}$. However, the ideal $L_2(\Delta)^{sq}=(x_1x_2x_4,x_1x_3x_4)$ is not squarefree Borel.
Following the proof of Proposition \ref{borel}, one observes that anyway $I(\Delta^{(j)})^\vee$ is  always squarefree Borel ideal for all $j$ no matter what is the cardinality of the Borel generators.

\medskip
Let $\B$ be a  Borel set not necessarily principal, and let $\Delta$ be the simplicial complex with  $\F(\Delta)=\B$. In \cite{FMS}, the authors describe 1-covers of $\Delta$ by using Theorem~\ref{BorelAlexander} and the fact that $(I+J)^\vee=I^\vee\cap J^\vee$ for all squarefree ideals $I$ and $J$. With similar argument, one can use Proposition \ref{B-generators} to have squarefree $k$-covers of $\Delta$ when $\Delta$ is pure. More precisely, let $\F(\Delta)=B( F_1,\ldots,F_m)$ be a Borel set and $|F_i|=|F_j|$ for all $i,j$. By Lemma~\ref{skeleton}, $\F(\Delta^{(d-k)})$ is also a Borel set with the Borel generators as described in this lemma. Now using the above mentioned fact, i.e. $(I+J)^\vee=I^\vee\cap J^\vee$ for all squarefree ideals $I$ and $J$, and the fact $I(\Delta^{(d-k)})^\vee=L_k(\Delta)^{sq}$, one has the result in more general case.

For example, let $\Delta$ be the simplicial complex with $\F(\Delta)=B(\{1,4,5\},\{2,3,4\})$. We find the 2-covers of $\Delta$ as follows: for the simplicial complex $\Delta_1$ with $\F(\Delta_1)=\{1,4,5\}$, Proposition~\ref{B-generators} yields
\begin{eqnarray*}
L_2(\Delta_1)^{sq}&=&(x_H\:\; H\in B(\{1,2,3,4\},\{2,3,4,5\})\;)\\
&=&(x_1x_2x_3x_4,x_1x_2x_3x_5,x_1x_2x_4x_5,x_1x_3x_4x_5,x_2x_3x_4x_5),
\end{eqnarray*}
and similarly for the simplicial complex $\Delta_2$ with $\F(\Delta_2)=\{2,3,4\}$, one obtains
\begin{eqnarray*}
L_2(\Delta_2)^{sq}&=&(\{x_H\:\; H\in B(\{1,2,3\},\{2,3,4\})\})\\
&=&(x_1x_2x_3,x_1x_2x_4,x_1x_3x_4,x_2x_3x_4).\hspace{3.2cm}
\end{eqnarray*}
Hence
\begin{eqnarray*}
 L_2(\Delta)^{sq}&=&L_2(\Delta_1)^{sq}\cap L_2(\Delta_2)^{sq}\\
 &=&(x_1x_2x_3x_4,x_1x_2x_3x_5,x_1x_2x_4x_5,x_1x_3x_4x_5,x_2x_3x_4x_5).
\end{eqnarray*}
Observe that the generators of $B(\Delta)$ as described in Proposition~\ref{B-generators} are not necessarily the minimal ones.
\begin{Theorem}
\label{borel-generators}
Let $\B=B( F)$ be a principal Borel set with Borel generator  $F=\{i_1<i_2<\cdots < i_d\}$, and let $\Delta$ be the simplicial complex with  $\F(\Delta)=\B$. Then $B(\Delta)=A(\Delta)$.
\end{Theorem}
\begin{proof}
We show that for every non-squarefree cover $\cb$ of $\Delta$ of positive order, there exists a decomposition $\cb=\ab+\bb$ such that $\ab$ is a squarefree cover of $\Delta$. This will imply that $B(\Delta)=A(\Delta).$ So consider a non-squarefree $k$-cover $\cb$ of $\Delta$ with $k>0$ and let $C=\supp(\cb)$. We may assume that $k$ is the maximum order of $\cb$, that is, if $\ell> k$, then $\cb$ is not an $\ell$-cover. Indeed, assume that $k$ is the maximum order of $\cb$ and  $\cb=\ab+\bb$ is a decomposition of $\cb$ such that $\ab$ and $\bb$ are covers of $\Delta$ of orders $i$ and $j$, respectively, with $k=i+j$. Then for every $k'<k$, one can choose $i'\leq i$ and $j'\leq j$ such that $i'+j'=k'$. Hence $\cb=\ab+\bb$ can be also considered a decomposition of $\cb$ as a $k'$-cover.

We denote by $\A_{\ell}$  the Borel set $ B(\{q,q+1,\ldots i_{\ell+q-1}\}\:\;    q=1,\ldots,d-\ell+1)$, for $\ell=1,\ldots,d$. Then by Proposition~\ref{B-generators}, for every $H\in \A_{\ell}$, the (0,1)-vector $x^{\mathbf{h}}$ with $\supp(\mathbf{h})=H$ is a squarefree $\ell$-cover of $\Delta$.\medskip
\\
{\em Step 1} (Defining the cover $\ab$)\textbf{.}
Let
$$\T=\{H\subseteq C\:\; H\in \A_{\ell} \text{ for some } \ell \}.$$  Since $\cb$ is a cover of positive order, it is  a $1$-cover of $\Delta$ as well. Hence Theorem \ref{BorelAlexander} implies that $\T$ is nonempty. Let
$$r=\max\{\ell \:\; \text{there exists } H\in \T \text{ with } H\in \A_{\ell} \text{ for some } \ell\}.$$
Observe that $r\leq k$ because  $k$ is the maximum order of $\cb$. Let $A$ be an element of $\T$ such that $A\in \A_{r}$. Consider the vector $\ab$ with $\supp(\ab)=A$. Then the vector $\ab$ is an $r$-cover of $\Delta$. We will show later that $\bb=\cb-\ab$ is a $(k-r)$-cover of $\Delta$ to obtain the desired decomposition of $\cb$.
\medskip
\\
{\em Step 2.} If $r\neq d$, we observe that there exist at least $t$ entries $c_j$ of $\cb$ with $c_j=0$ where $j\leq i_{r+t}$ for all $t=1,\ldots d-r.$ In fact, if at most $t-1$  number of  entries $c_j$ with $j\leq i_{r+t}$ are zero, then there exists a subset $A'$ of $C$ which precedes $\{t,t+1,\ldots,i_{r+t}\}$. This means that $A'\in \A_{r+1}$, a contradiction to the choice of $r$.
\medskip
\\
{\em Step 3.}
We  show that $\bb=\cb-\ab$ is a $(k-r)$-cover of $\Delta$. If $r=d$, then $A=[i_d]$ and $\ab$ is the vector for which all entries are $1$. Thus $\bb$ is a $(k-r)$-cover of $\Delta$ because every facet of $\Delta$ is of cardinality $d$. Therefore, we have the desired decomposition of $\cb$. So we may assume that $r<d$.

 We consider a facet $G$ of $\Delta$, and show that $\sum_{i\in G}b_i\geq k-r$. For this purpose, attached to $G$, we inductively define a sequence $G=G_0,G_1,\ldots,G_m$ of facets of $\Delta$, where $m=|G\cap A |-r$,  with the following properties:
\begin{enumerate}
\item[(i)] $1+|G_i\cap A|=|G_{i-1}\cap A|$ for all $i=1,\ldots,m;$
\item[(ii)] $1+\sum_{j\in G_i}c_j\leq \sum_{j\in G_{i-1}}c_j $ for all $i=1,\ldots,m.$
\end{enumerate}
Assume that the facets $G_0,\ldots,G_{\ell}$ is already defined for some $\ell<m$. Let $G_{\ell}=\{j_1<\cdots<j_d\}$ and $j_s=\max\{j_i\in G_{\ell}\: \; j_i\in A\}$. By  property (i) and since $\ell<m=|G\cap A|-r$, one has $|G_{\ell}\cap A|>r$. Hence $s>r$. By step 2, since $\cb$ has at least $d$ zero entries, the set $\{i\:\; c_i=0 \text{ and } i\not\in G_{\ell}\}$ is nonempty. In fact, otherwise  $i\in G_\ell$ whenever $c_i=0$, and so $G_{\ell}\cap \supp(\cb) = \emptyset$, a contradiction. Let
$$t=\min\{i\:\; c_i=0 \text{ and } i\not\in G_{\ell}\}.$$
 We set $G_{\ell+1}=(G_{\ell}\setminus \{j_s\})\cup \{t\}$. Since $c_t=0$, one has $t\not\in A$, and so (i) holds for $G_{\ell+1}$. But $c_{j_s}\neq 0$ because $j_s\in A$. Thus (ii) also holds for $G_{\ell+1}$. So we only need to show $G_{\ell+1}$ is a facet of $\Delta$.  If $t<j_s$, then $G_{\ell+1}$ is a facet of $\Delta$ by the definition of the Borel sets. Furthermore, since $t\not\in G_{\ell}$, one has $t\neq j_s$. So assume that $t>j_s$ and
$i_{s+q}<t\leq i_{s+q+1}$. Then for each $q'$ where $q'=1,\ldots,q+1$, we have $j_{s+q'}\leq i_{s+q'-1}$. By contrary, assume that $j_{s+q'}> i_{s+q'-1}$ for some $q'=1,\ldots,q$. Then exactly $s+q'-1$ elements of $G_{\ell}$ belongs to the set $[i_{s+q'-1}]$. On the other hand, by Step 2 there exist $s+q'-1$ entries $c_i=0$ with $i\in [i_{s+q'-1}]$. So by the choice of $t$, we would have $c_i=0$ for $i=j_1,\ldots,j_{s+q'-1}$ which implies $\{j_1,\ldots,j_{s+q'-1}\}\cap A=\emptyset$, a contradiction. Thus $j_{s+q'}\leq i_{s+q'-1}$ for all $q'=1,\ldots,q+1$, and hence the set $\{j_1<\cdots<\hat{j}_s<\cdots < j_{s+q+1}\}$ precedes $\{i_1<\cdots< i_{s+q}\}.$ Let $E_1=\{j'_{s+q+1}<\cdots<j'_d\}$ equals the set $\{t,j_{s+q+2},\ldots,j_d\}$. Considering the fact $t\leq i_{s+q+1}$, it follows that $E_1$ precedes $\{i_{s+q+1},\ldots , i_d\}$ because if $j_{q'}<t$,  then $j_{q'}\leq i_{q'-1}$ for $q'=s+q+2,\ldots,d$. Therefore, $G_{\ell+1}=\{j_1<\cdots<\hat{j_s}<\cdots < j_{s+q+1}<j'_{s+q+1}<\cdots<j'_d\}$ precedes $F$  which means $G_{\ell+1}$ is a facet of $\Delta$, as desired.

Now using the property (ii) of the sequence $G_0,\ldots,G_m$, we obtain
$$\sum_{i\in G}c_i\geq \sum_{i\in G_m}c_i+m \geq k+m=k+|G\cap A |-r.$$
The second above inequality holds because $\cb$ is a $k$-cover of $\Delta$. So
$$\sum_{i\in G}b_i=\sum_{i\in G}c_i-\sum_{i\in G}a_i\geq (k+|G\cap A |-r)-|G\cap A |=k-r.$$
Thus $\bb$ is a $(k-r)$-cover of $\Delta$, and this means that $\cb=\ab+\bb$ is the desired decomposition of $\cb$.
\end{proof}

The statement of Theorem \ref{borel-generators} does not hold for the Borel sets in general. Once more consider the example after Proposition \ref{borel} where $\Delta$ is the simplicial complex with $\F(\Delta)=B(\{1,4,5\},\{2,3,4\})$. Then $\cb=(2,1,1,1,0)$ is a 3-cover of $\Delta$. We have already known the squarefree 2-covers of $\Delta$ and one can also find the squarefree 1-covers and 3-cover of $\Delta$ by Proposition~\ref{B-generators} in order to see that $\cb$ cannot be decomposed to squarefree covers of $\Delta$. Hence $B(\Delta)\neq A(\Delta).$

\begin{Corollary}
Let $\B=B( F)$ be a principal Borel set with Borel generator  $F=\{i_1<i_2<\cdots < i_d\}$, and let $\Delta$ be the simplicial complex with  $\F(\Delta)=\B$.
Then $B(\Delta^{(j)})=A(\Delta^{(j)})$ for every $j=0,\ldots,d-1$.
\end{Corollary}
\begin{proof}
It is enough to notice that by Lemma~\ref{skeleton}, the set $\F(\Delta^{(j)})$ is a  principal Borel set. Hence Theorem~\ref{borel-generators} implies the result.
\end{proof}

An immediate consequence of Theorem \ref{borel-generators} is the following result in \cite[Proposition 4.6]{HHT}.
\begin{Corollary}
Let $\Sigma_n$ denote the simplex of all subsets of $[n]$. Then the $S$-algebra $A(\Sigma_n^{(d-1)})$ is minimally generated by the monomials $x_{j_1}x_{j_2}\cdots x_{j_{n-d+k}}t^k $, where $k=1,\ldots,d$ and $1\leq j_1<j_2<\cdots <j_{n-d+k}\leq n$.
\end{Corollary}
\begin{proof}
Consider the Borel set $\B=B(\{n-d+1,n-d+2,\ldots,n\})$. Then $\F(\Sigma_n^{(d-1)})=\B$. Thus by Proposition ~\ref{B-generators} and Theorem~\ref{borel-generators}, the $S$-algebra $A(\Sigma_n^{(d-1)})$ is generated by the monomials $x_Ht^k$ where
\[
H\in B(\{1,2,\ldots,n-d+k\},\{2,3,\ldots,n-d+k+1\},\ldots,\{d-k+1,d-k+2,\ldots,n\}),
\]
for $k=1,\ldots,d$. The above mentioned set is exactly the set of all subsets of $[n]$ of cardinality $n-d+k$. On the other hand, by these monomials we have a minimal system of generators. In fact for the case that $n=d$ there is nothing to prove, and for the case that $n>d$  by contrary, assume that for a generator $x_Ht^k$, there exist monomials $x_{H_1}t^{k_1}$ and $x_{H_2}t^{k_2}$ in this set such that $x_{H_1}x_{H_2}|x_H$ and $k_1+k_2=k$. Hence we have
\[
(n-d+k_1)+(n-d+k_2)=|H_1|+|H_2|\leq |H|=(n-d+k).
\]
Thus $n\leq d$, a contradiction.
\end{proof}

The following proposition exhibit a condition which guarantees the existence of a generator of degree $\dim(\Delta)+1$ in the minimal set of monomial generators of $A(\Delta)$.
\begin{Proposition}
\label{higher-generator}
Let $\B=B( F)$ be a  principal Borel set with Borel generator  $F=\{i_{1}<i_{2}<\cdots < i_{d}\}$,  and let $\Delta$ be the simplicial complex with  $\F(\Delta)=\B$. Then $x_1x_2\cdots x_{i_d}t^d$ belongs to the minimal set of monomial generators of $A(\Delta)$ if and only if  $i_1\neq 1$.
\end{Proposition}
\begin{proof}
First observe that $i_1\neq 1$ if and only if $i_j\neq j$ for all j. So we show that the $d$-cover $\cb$ of $\Delta$ with $c_i=1$ for all $i$, is decomposable if and only if $i_j=j$ for some $j$. First, suppose that $\cb$ is decomposable and $\cb=\ab+\bb$ is a decomposition of $\cb$. Since $\Delta$ is pure of dimension $d-1$, the vector $\cb$ is the only squarefree $d$-cover of $\Delta$. Therefore, the order of $\ab$ and $\bb$ is nonzero. Let the order of $\ab$  be $r$ and the order of $\bb$ be $d-r$. Since $\cb$ does not have a decomposition to a 0-cover and a $d$-cover, the same is true for the covers $\ab$ and $\bb$. Hence if $A=\supp(\ab)$ and $B=\supp(\bb)$, then they are  in the form described in Proposition~\ref{B-generators}. In other words, there exist numbers $j$ and $l$ such that  $A$ precedes  $\{l-r+1,l-r+2,\ldots,i_l\}$ and $B$ precedes  $\{j-d+r+1, j-d+r+2, \ldots, i_j\}$. Observe that $[i_d]$ is the disjoint union of  $A$ and $B$ because $\cb=\ab+\bb$. Thus we may assume that $i_d\in A$ which implies $l=d$. Moreover, we see that $|A|+|B|= i_d$. For this reason, since $|A|=i_d-(d-r+1)+1$ and $|B|=i_j-(j-d+r+1)+1$, we obtain $i_j=j$.

Conversely, let $i_j=j$ for some $j$. By Theorem~\ref{B-generators}, the monomial $x_At^j\in L_j(\Delta)$ where $A=\{1,2,\ldots,i_j=j\}$ and  the monomial $x_Bt^{d-j}\in L_{d-j}$ where $B=\{j+1,j+2,\ldots,i_d\}$. Let $\ab$ and $\bb$ be the vectors with $A=\supp(\ab)$ and $B=\supp(\bb)$. Then we have the decomposition  $\cb=\ab+\bb$ of the cover $\cb$.
\end{proof}

\section*{Acknowledgements}
This paper was completed while the first author was visiting Universit\"at Duisburg-Essen, Campus Essen. The authors would like to thank  Professor J\"{u}rgen Herzog
for his hospitality, support, and stimulating discussions.

{}


\begin{thebibliography}{}


\bibitem{Cocoa} CoCoA Team, CoCoA: a system for doing computations in commutative algebra, Available at http:
//cocoa.dima.unige.it.

\bibitem{DV} L.\ A.\ Dupont, R.\ H.\ Villarreal, Symbolic Rees algebras, vertex covers and irreducible
representations of Rees cones, Algebra Discrete Math. 10 (2010), no. 2, 64-–86.

\bibitem{VHF} V.\ Ene, J.\ Herzog, F.\ Mohammadi,
Monomial ideals and toric rings of Hibi type arising from a finite poset,
European J. Combin. 32 (2011), no. 3, 404–-421.

\bibitem{FMS} C.\ Francisco, J.\ Mermin, J.\ Schweig,
Borel generators,
J. Algebra 332 (2011), 522-–542.

\bibitem{GRV} I.\ Gitler, E.\ Reyes, R.\ H.\ Villarreal,
Blowup algebras of square-free monomial ideals and some links to combinatorial optimization problems,
Rocky Mountain J. Math. 39 (2009), no. 1, 71–-102.

\bibitem{GVV} I.\ Gitler, C.\ Valencia, R.\ H.\ Villarreal,
A note on Rees algebras and the MFMC property,
Beiträge Algebra Geom. 48 (2007), no. 1, 141–-150.

\bibitem{HHBook} J.\ Herzog, T.\ Hibi,  Monomial ideals. Graduate Texts in Mathematics, 260. Springer-Verlag London, Ltd., London, 2011.

\bibitem{HHT} J.\ Herzog, T.\ Hibi, N.\ V.\ Trung, Symbolic powers of monomial ideals and vertex cover algebras,
Adv. Math. 210 (2007), no. 1, 304–-322.


\bibitem{XIN} J.\ Herzog, T.\ Hibi, N.\ V.\ Trung, X. Zheng,
Standard graded vertex cover algebras, cycles and leaves,
Trans. Amer. Math. Soc. 360 (2008), no. 12, 6231–-6249.


\bibitem{HSV} C.\ Huneke, A.\ Simis, W.\ Vasconcelos,
Reduced normal cones are domains, Invariant theory (Denton, TX, 1986), 95–-101,
Contemp. Math., 88, Amer. Math. Soc., Providence, RI, 1989.


\bibitem{SVV} A.\ Simis, W.\ Vasconcelos, R.\ H.\ Villarreal,
On the ideal theory of graphs,
J. Algebra 167 (1994), no. 2, 389–-416.


\end{thebibliography}
\end{document}